\documentclass{article}
 \usepackage{a4wide}
\usepackage{amssymb,amsmath}
\usepackage{amsthm}

\usepackage[utf8]{inputenc}

\usepackage{graphicx}
\usepackage{color}
\usepackage{subfig}
\usepackage{caption}
\usepackage{wrapfig}

\usepackage{enumitem}
\setlist{noitemsep, topsep=0pt}

\newtheorem{theorem}                                {Theorem}
\newtheorem{lemma}                    {Lemma}

\theoremstyle{definition}

\newtheorem*{conjecture}              {Conjecture}

\theoremstyle{remark}

\usepackage[OT1]{fontenc}
\DeclareFontFamily{OT1}{pzc}{}
\DeclareFontShape{OT1}{pzc}{m}{it}{<-> s * [1.200] pzcmi7t}{}
\DeclareMathAlphabet{\mathscr}{OT1}{pzc}{m}{it}

\newcommand{\CN}{\ensuremath{N}} 
\newcommand{\CE}{\ensuremath{\mathscr E}} 
\newcommand{\CB}{\ensuremath{\mathscr B}}

\newcommand{\PP}{\mathcal{P}}

\newcommand{\N}{\mathbb{N}}

\newcommand{\R}{\mathbb{R}}

\begin{document}
\title{Isoperimetric Inequalities on Hexagonal Grids }
\author{ Berit Gru\ss ien\thanks{
Humboldt-Universität zu Berlin, 
Work by Gru\ss ien was supported by the Deutsche Forschungsgemeinschaft (DFG)
within the research training group "Methods for Discrete Structures"
(GrK 1408).
}
}
\maketitle

\begin{abstract}
\begin{quote}
We consider the edge- and vertex-isoperimetric probem on finite and infinite hexagonal grids:
For a subset $W$ of the hexagonal grid of given cardinality, 
we give a lower bound for the number of edges between $W$ and its complement, 
and lower bounds for the number of vertices in the neighborhood of $W$ and for the number of vertices in the boundary of $W$.
For the infinite hexagonal grid the given bounds are tight.
\end{quote}
\end{abstract}

\section{Introduction}
Let us consider sets of points in the continuous plane. 
To each set $W$ of points we can assign its area~$a(W)$ and its perimeter~$p(W)$ which is the boundary of set $W\!$. 
Then, the isoperimetric problem is to 
consider all possible sets of points with a special fixed area~$a_0$, 
and determine the minimum length of boundary 
a set with area $a_0$ can have.
In the continuous plane the answer to this problem is already known. The set with minimum boundary for fixed area $a_0$
has always the shape of a disk, and the following inequality holds for all sets $W$:
$$p(W)\geq 2\sqrt{\pi}\cdot \sqrt{a(W)}.$$
Inequalities of this form are called isoperimetric inequalities.

Now, we can look at the same problem in a discrete setting.
Given a graph $G=(V,E)$ and a finite subset $W$ of vertices of $G$ we define the area of $W$
as the number $|W|$ of vertices in $W$.
For the perimeter of $W$, there are different notions.
Common notions use the number of neighbor vertices or the number of boundary vertices for the perimeter.
It is also reasonable to consider the number of outgoing edges
to measure the perimeter.

Figure \ref{fig:perimetermeas} illustrates these different notions.
\begin{figure}[ht]
\centering
\begin{minipage}{\textwidth}
\subfloat[Subset $W$ of vertices]{\includegraphics[totalheight=2.3cm]{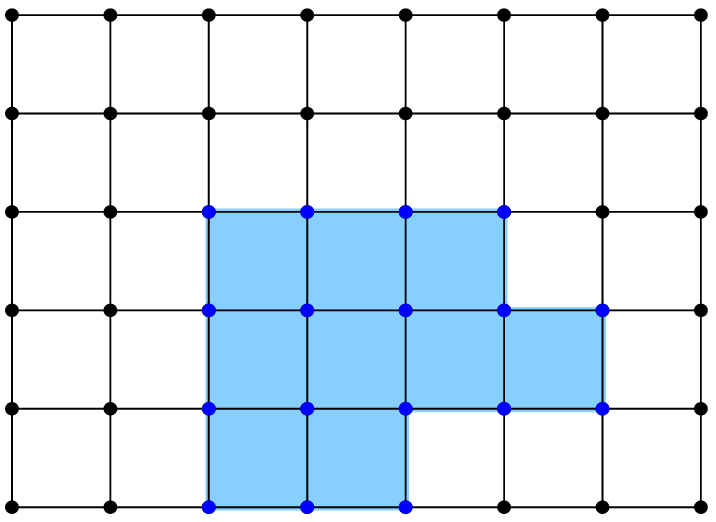}}\hfill
\subfloat[Neighbor vertices]{\includegraphics[totalheight=2.3cm]{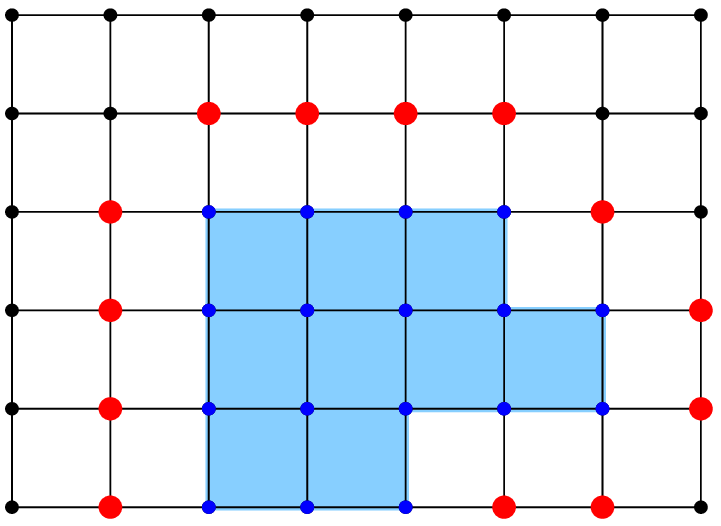}}\hfill
\subfloat[Boundary vertices]{\includegraphics[totalheight=2.3cm]{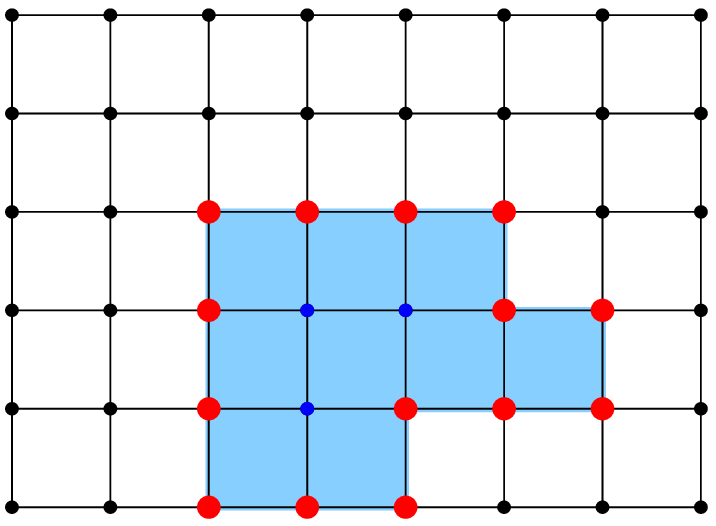}}\hfill
\subfloat[Outgoing edges]{\includegraphics[totalheight=2.3cm]{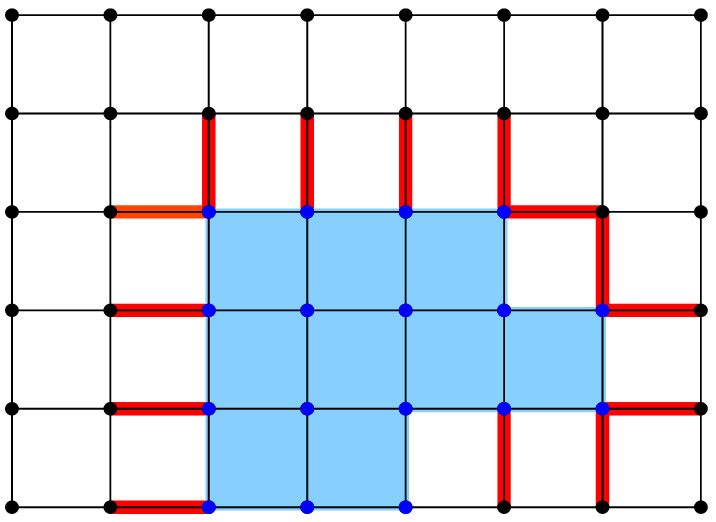}}
\end{minipage}
\caption{Different perimeter measures}
\label{fig:perimetermeas}
\end{figure}

Depending on the graphs we consider, each notion might define a different isoperimetric problem.
The isoperimetric problem is called (neighbor or boundary) vertex-isoperimetric problem if one of the first two notion is used for the perimeter,
and edge-isoperimetric problem if the last notion is used.

One often considers isoperimetric problems and searches for isoperimetric inequalities on graphs that have a uniform structure.
One of the first and most important results 
is that of Harper~\cite{harper-vertex} in 1966 who solved the neighbor vertex-isoperimetric problem 
on the discrete cube $\{0,1\}^n=\PP(\{1,\dots,n\})$, 
where two sets $A$ and $B$ are joined by an edge if $|A\bigtriangleup B|=1$.
A solution for the edge-isoperimetric problem on the discrete cube was given by 
Harper~\cite{harper-edge}, Lindsey~\cite{lindsey}, Bernstein~\cite{bernstein} and Hart~\cite{hart}.
These results were extended to finite grids $[k]^n$ by  Bollob\'as and Leader~\cite{bollobasleader-vertex,bollobasleader-edge}.
Wang and Wang provided a solution 
for the neighbor vertex-isoperimetric problem on the infinite grid~$\mathbb{Z}^n$~\cite{wangwang-discrete,wangwang-extremal}.
The last result we want to mention is that of Gravier 
who solved the  (boundary) vertex-isoperimetric problem on the infinite triangular grid~\cite{gravier}.

In this article, we consider infinite and finite hexagonal grids and give isoperimetric inequalities for the mentioned notions of perimeter.
For the inifinite hexagonal grid the inequalities are tight.

\section{Prelimitaries}

A \emph{graph} is a pair $G=(V(G),E(G))$, where $V(G)$ is a (not necessarily finite) set, the vertex set, 
and $E(G)$ is a binary relation on $V(G)$, the edge relation.
Let $G$ be a graph and $W$ be a finite subset of $V(G)$.
We define the \emph{area} of $W$ as the cardinality $|W|$ of $W$, 
and introduce three notions for the \emph{perimeter} of $W$:
the number of neighbor vertices, 
where the set of neighbor vertices of $W$ is the set $\CN (W):=\{x\in V\setminus W\mid \{x,y\}\in E \land y\in W\}$,
the cardinality of the set of boundary vertices $\CB (W):=\{x\in W\mid \{x,y\}\in E \land y\in V\setminus W\}$ and 
the number of outgoing edges
$\CE (W):=\{\{x,y\}\in E\mid x\in W \land y\in V\setminus W\}$ (see Figure \ref{fig:perimetermeas} for an illustration).

The \emph{infinite hexagonal grid} is the grid graph formed by a tiling of regular hexagons,
in which three regular hexagons meet at each vertex.
We define the \emph{finite hexagonal grid of radius $r$}, denoted by $G_r$, 
inductively as a subgraph of the infinite hexagonal grid:
For $r=1$, the hexagonal grid~$G_1$, is a cycle of length $6$, depicted by a regular hexagon.
The graph $G_r$  consists all vertices and edges of $G_{r-1}$ and  
the layer of regular hexagons surrounding $G_{r-1}$ (see Figure~\ref{fig:finitegrid}).
We also denote the infinite hexagonal grid by $G_{\infty}$.

\begin{figure}[ht]
\centering
\includegraphics[width=0.22\textwidth]{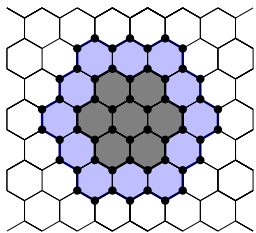}
\caption{$G_2$~and the layer~of~hexagons surrounding~it}
\label{fig:finitegrid}
\end{figure}

%


The following lemma can be shown by an easy induction.
\begin{lemma}\label{lem:NbOfVerticesOfFinHexGrid}
   Let $G_\infty$ be the infinite hexagonal grid and the subgraph $G_r$ be the finite hexagonal grid of radius $r$. 
  Then $|\CN(V(G_r))|=6r$,  $|\CE(V(G_r))|=6r$ and $|V(G_r)|=6r^2$.
\end{lemma}

\section{Infinite Hexagonal Grid}
Let $G_{\infty}$ be the infinite hexagonal grid.
\begin{theorem}\label{thm:infinitegrid}
 For any finite subset $W\subset V(G_{\infty})$ we have
	\begin{enumerate}
	   \item\label{lab:thm:infinitegrid1} $|\CN(W)|\geq c\cdot \sqrt{|W|}$ for $c=\sqrt{6}$, and $c$ is optimal.
		\item\label{lab:thm:infinitegrid2} $|\CE(W)|\geq c\cdot \sqrt{|W|}$ for $c=\sqrt{6}$, and $c$ is optimal.
		\item\label{lab:thm:infinitegrid3} $|\CB(W)|\geq c\cdot \sqrt{|W|+\frac{c^2}4}-\frac{c^2}2$ for $c=\sqrt{6}$, and $c$ is optimal.
	\end{enumerate}
\end{theorem}

We begin with the proof of Theorem~\ref{thm:infinitegrid}.\ref{lab:thm:infinitegrid1}, which consists of four parts.
First, we transform our set $W$ into a more suitable one having the same cardinality and at most as many neighbor vertices.
Then we establish a lower bound for the number of neighbor vertices, and after that, an upper bound for the number of vertices in $W\!$.
Finally, we combine these bounds to obtain the inequality in  Theorem~\ref{thm:infinitegrid}.\ref{lab:thm:infinitegrid1}.

\paragraph*{Part 1 \textnormal{(Transforming $W$)}.}
Let us fix three non-zero vectors $\vec{v}_1$, $\vec{v}_2$ and $\vec{v}_3$ 
in the tiling of regular hexagons,
such that for each edge of the tiling 
there is a vector that 
is parallel to this edge.
We call the vectors $\vec{v}_1$, $\vec{v}_2$ and $\vec{v}_3$ also \emph{direction vectors}, or short \emph{directions},
and talk of direction $1$, $2$ and $3$ instead of $\vec{v}_1$, $\vec{v}_2$ and $\vec{v}_3$.
If we remove all edges parallel to a direction vector we obtain  
a partition of the vertices of the grid by considering the connected components.
We call such a connected component a \emph{row}.
Thus, a direction vector partitions the infinite hexagonal grid into disjoint \emph{rows} of vertices 
(see Figure~\ref{fig:directions} for an illustration).
Further, it is easy to see that the intersection of two rows of different directions 
always consists of two adjacent vertices.
\begin{figure}[ht]
\centering
\begin{minipage}{0.8\textwidth}
\subfloat[Direction $1$ (blue)]{\includegraphics[totalheight=3.3cm]{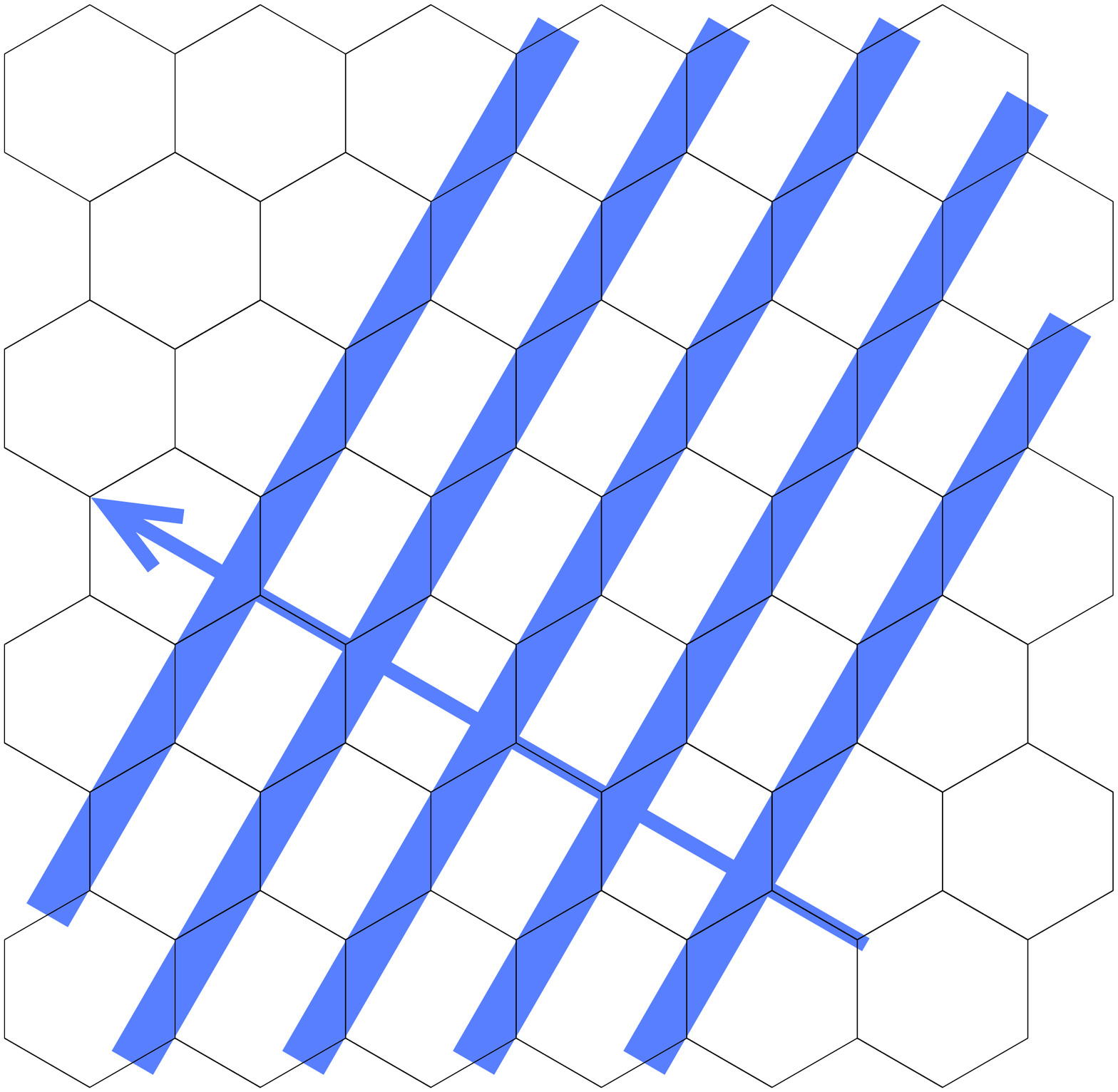}}\hfill
\subfloat[Direction $2$ (red)]{\includegraphics[totalheight=3.3cm]{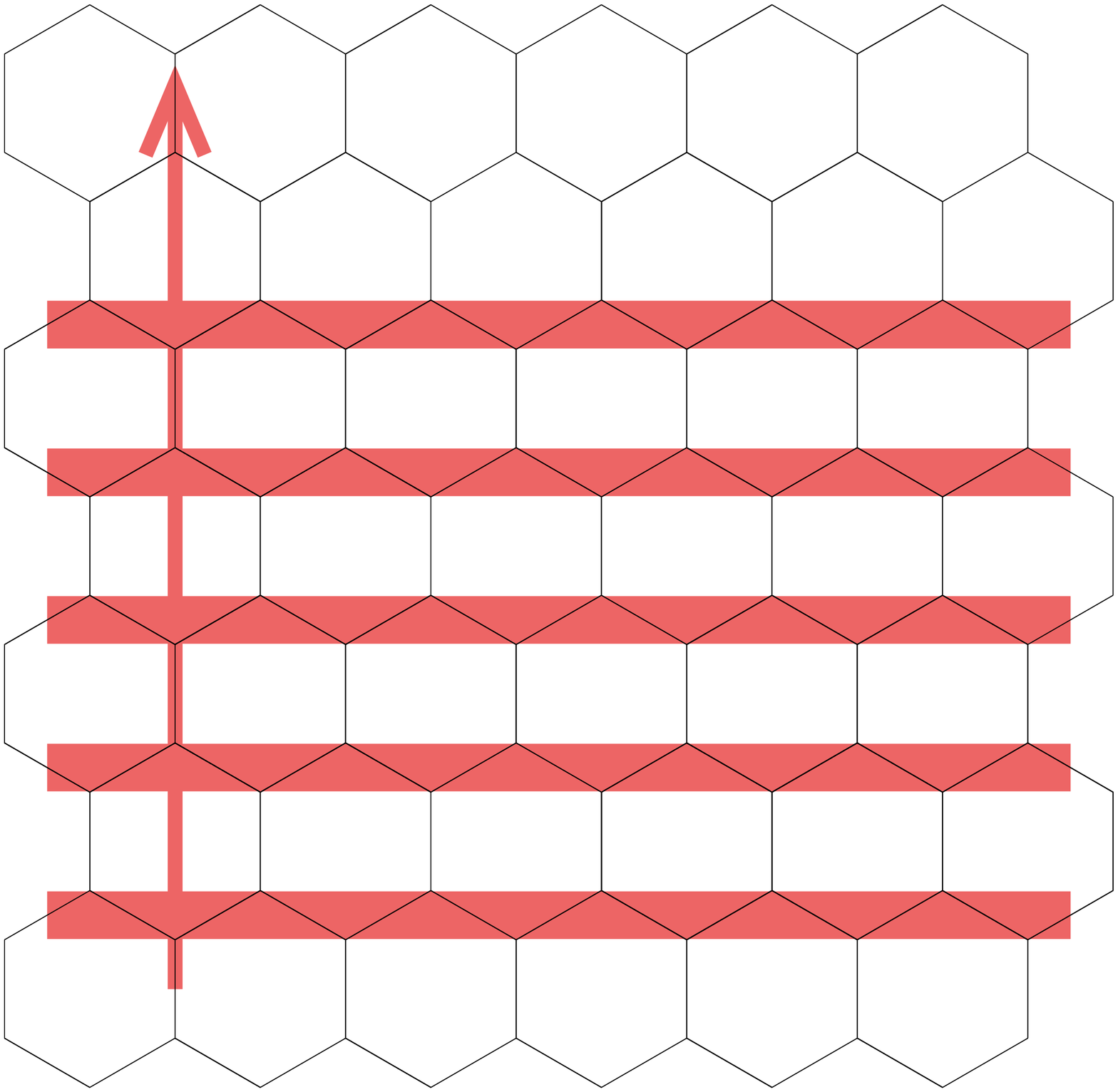}}\hfill
\subfloat[Direction $3$ (green)]{\includegraphics[totalheight=3.3cm]{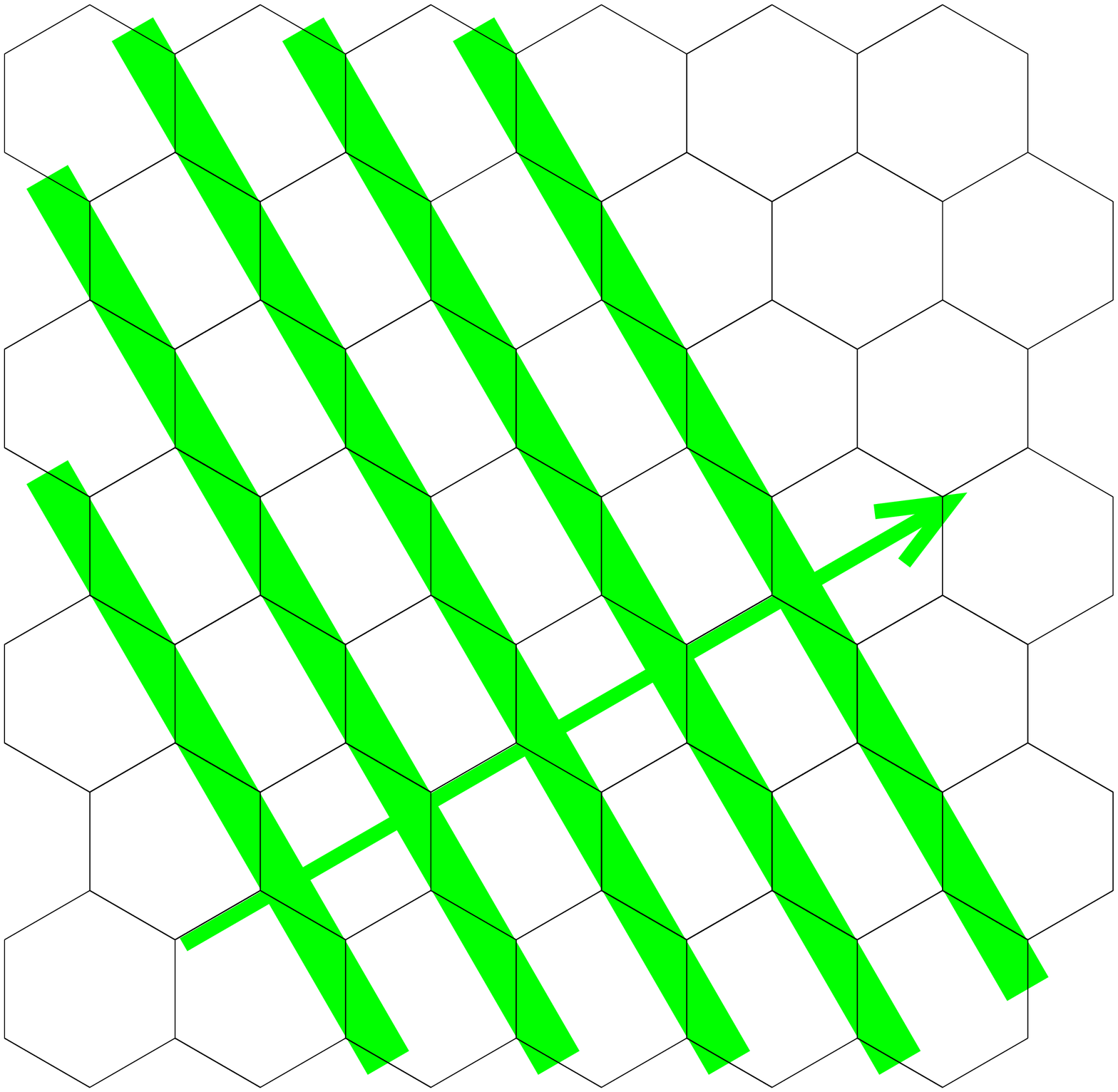}}
\end{minipage}
\caption{The $3$ different directions and their rows}
\label{fig:directions}
\end{figure}

Let $Z$ be a finite subset of $V(G_{\infty})$.
A vertex of $G_{\infty}$ is \emph{black with respect to $Z$} if it belongs to  $Z$, and \emph{white with respect to $Z$} if it does not belong to  $Z$.
Further, a row is \emph{white with respect to $Z$} if no vertex of it belongs to $Z\!$, and 
if at least one vertex of a row belongs to $Z$, then this row is \emph{gray with respect to $Z$}.
To simplify matters we will just talk of \emph{black} and \emph{white} vertices, and \emph{white} and \emph{gray} rows,
and make sure that it is clear from the context which set $Z$ we are refering to. 
Further, we call a row of direction $i$ \emph{bad} if it is white and separates two gray rows of direction $i$.


Let us consider a finite subset $W\subset V(G_{\infty})$.
With respect to $W\!$, we can assume that there exists no bad rows for any direction:
\begin{lemma}
For each set $W\!$, we can find a set $W'$
of same cardinality with less or equally many neighbor vertices having no bad rows 
for any direction. 
\end{lemma}

\begin{figure}[hbt]
\centering
\begin{minipage}{1\textwidth}
\subfloat[A white row of direction~$1$ partitioning $W$ into $W_1$ and~$W_2$]
{\label{fig:whiterowsa}\includegraphics[totalheight=4cm]{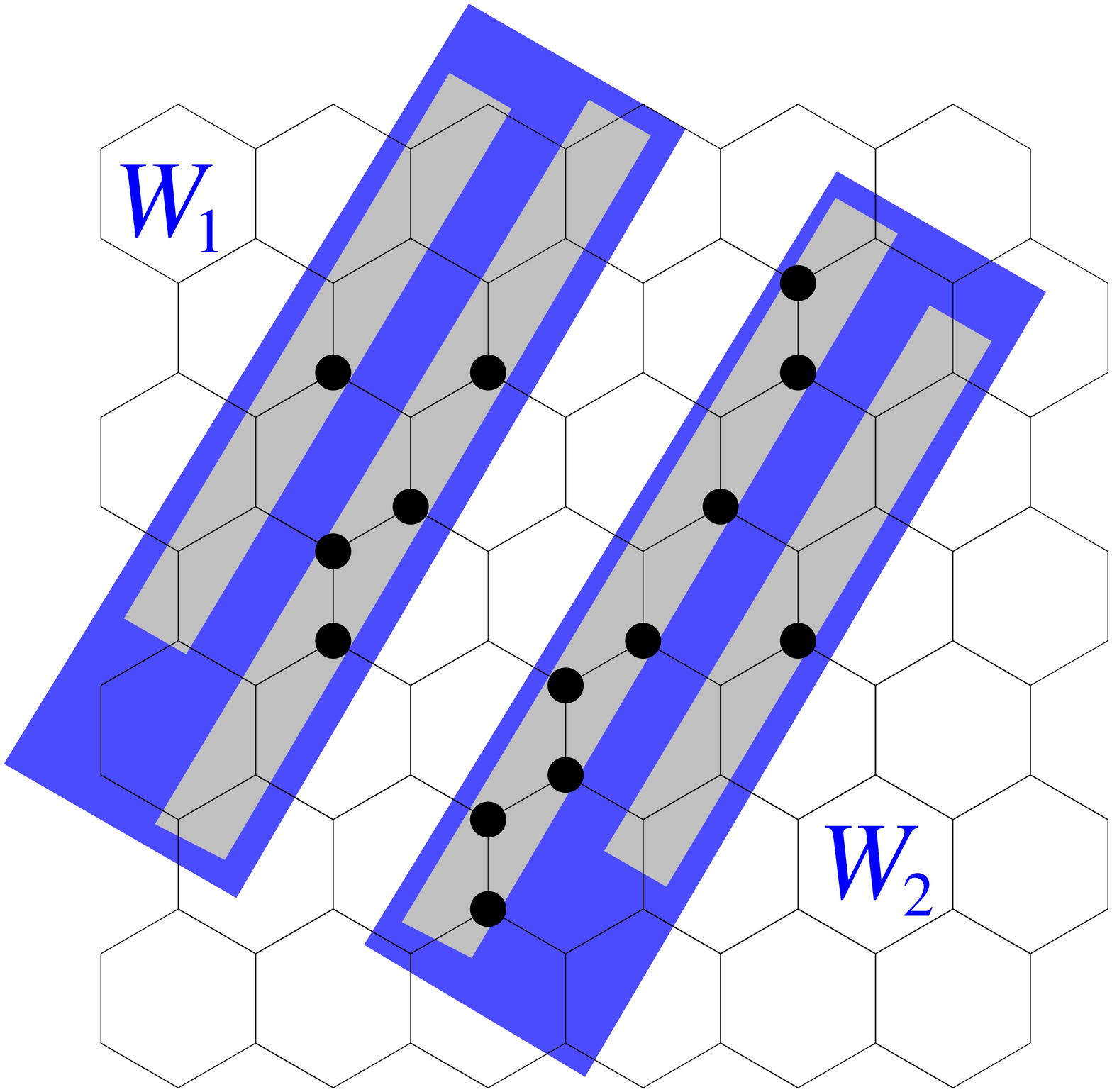}}\hfill
\subfloat[Neighbor vertices of $W_1$ and $W_2$]
{\label{fig:whiterowsb}\includegraphics[totalheight=4cm]{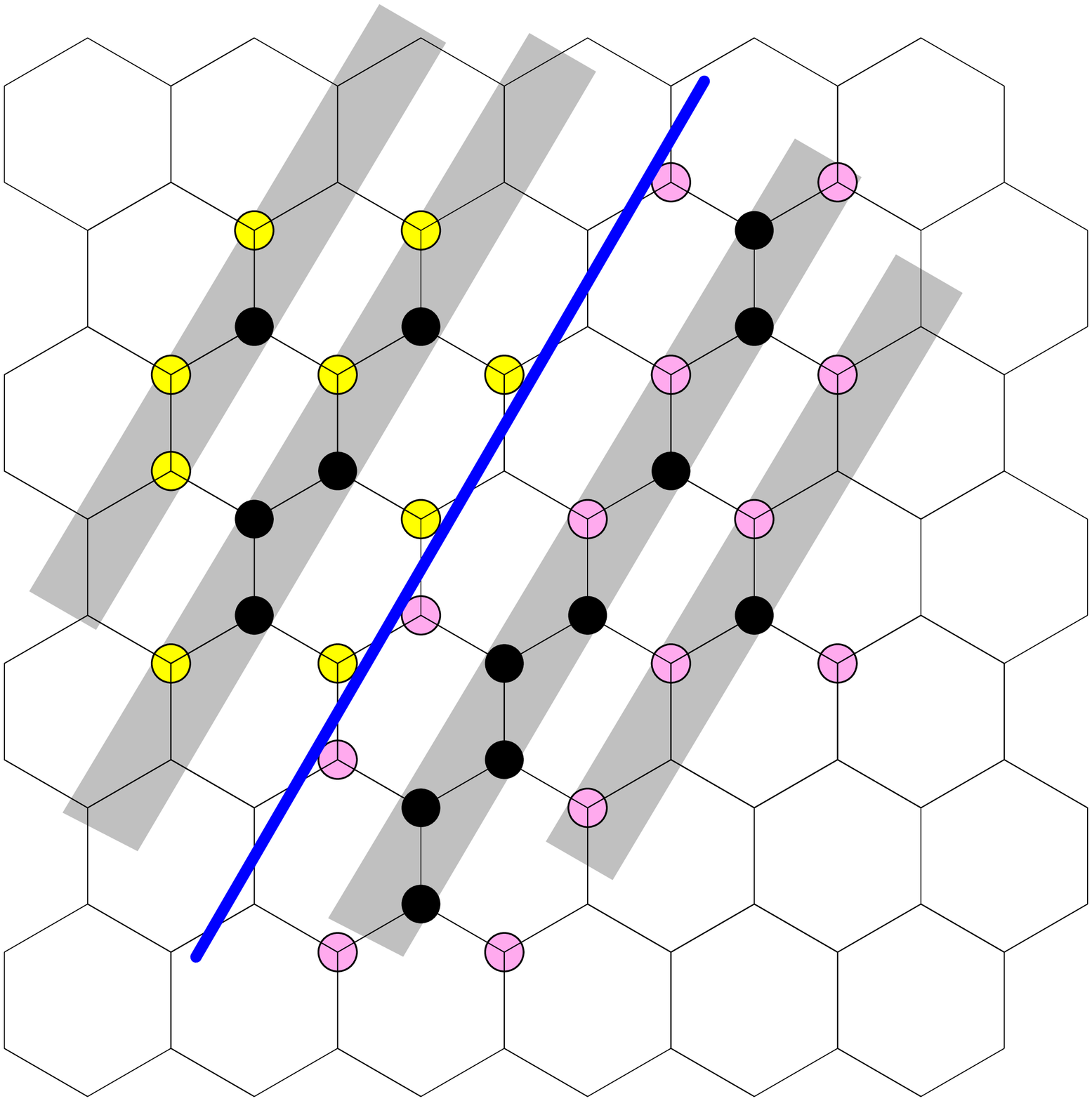}}\hfill
\subfloat[Movement along rows of direction $2$ to eliminate white row of direction 1]
{\label{fig:whiterowsc}\includegraphics[totalheight=4cm]{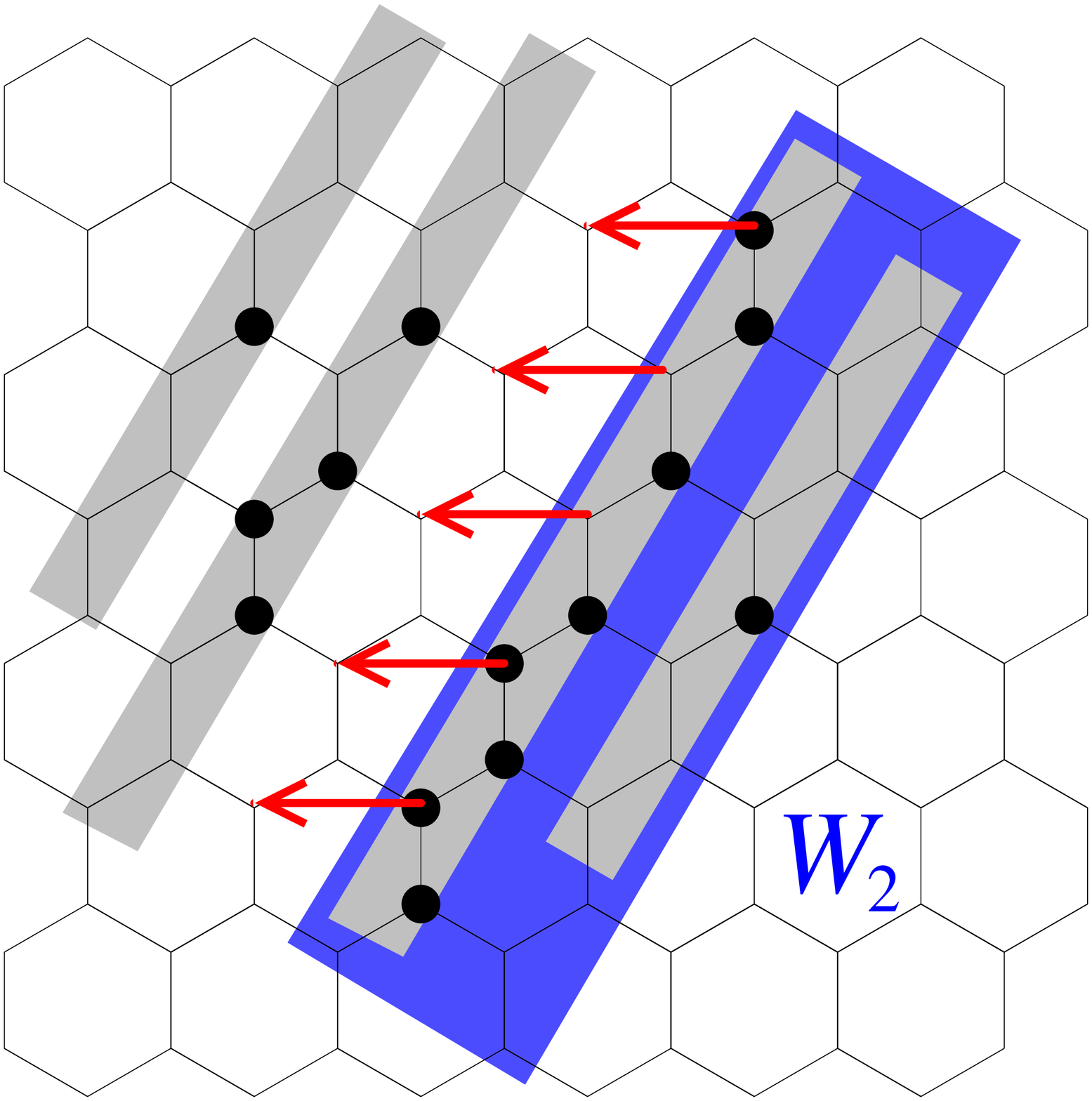}}
\end{minipage}
\caption{Eliminating white rows}
\label{fig:whiterows}
\end{figure}

\begin{proof}
Let us consider the set $W\!$.
Assume in direction $i\in\{1,2,3\}$ there is a bad row $R_i$. 
Then this white row partitions the set $W$ into two sets $W_1$ and $W_2$ (see Figure~\ref{fig:whiterowsa}).
Due to the structure of the grid the sets $\CN(W_1)$ and $\CN(W_2)$ of neighbor vertices are disjoint sets of vertices,
which can easily be perceived by Figure~\ref{fig:whiterowsb}.
Therefore, it is possible to shift the vertices of $W_2$ 
closer to the vertices of $W_1$
without increasing the number of neighbor vertices.
We do this by fixing some $j\in\{1,2,3\}\setminus\{i\}$ and moving all black vertices $v\in W_2$ along their row of direction $j$ 
two vertices closer to $R_i$
(see Figure~\ref{fig:whiterowsc}).
By \emph{moving a vertex} we mean replacing it
by its corresponding shifted vertex. 
Notice that this way, we do not produce any bad rows for direction~$j$.
Thus, it is possible to eliminate bad rows of direction $i$ without creating new bad rows in direction $j$.
In short, we also say we \emph{eliminate} a bad row of direction $i$ (by \emph{moving} $W_2$) \emph{agreeable} to direction $j$.

%


Let $Z$ be an arbitrary non-empty finite subset of vertices of $G_\infty$.
Let $R^i_1$ and $R^i_2$ be the two outermost white rows of direction $i$ that have a black vertex in its neighborhood.
We define the \emph{parallelogram} of a subset $Z$ 
to be the only connected component of $G_\infty\setminus {(R^1_1\cup R^1_2\cup R^2_1\cup R^2_2)}$ that is finite,
and denote it by $P_Z$
(see Figure~\ref{fig:parallelwhite_a}).


Now, we can inductively construct the set $W'$ from $W\!$.
First we eliminate all bad rows (one after another) of 
direction $1$ agreeable to direction $2$, and after that,
all bad rows of direction $2$ agreeable to direction $1$.
We obtain a set $W^*\!$, for which, in case there was at least one bad row of direction~$1$ or~$2$, the 
parallelogram $P_{W^*}$ is a proper subset of the parallelogram of $W\!$.
Now there might be bad rows regarding direction~$3$. 
We can eliminate each such bad row $R_3$ by moving the vertices of a smallest component of $P_{W^*}\setminus R_3$
agreeable to the direction of the rows that contain a longest side of the parallelogram (see Figure~\ref{fig:parallelwhite_b} for an illustration).
As a consequence the parallelogramm of the set we obtain is a subset of $P_{W^*}$.

%
%

\begin{figure}[ht]
\centering
\begin{minipage}{0.8\textwidth}
\subfloat[Parallelogram of a subset $Z$]
{\label{fig:parallelwhite_a}\includegraphics[totalheight=3.4cm]{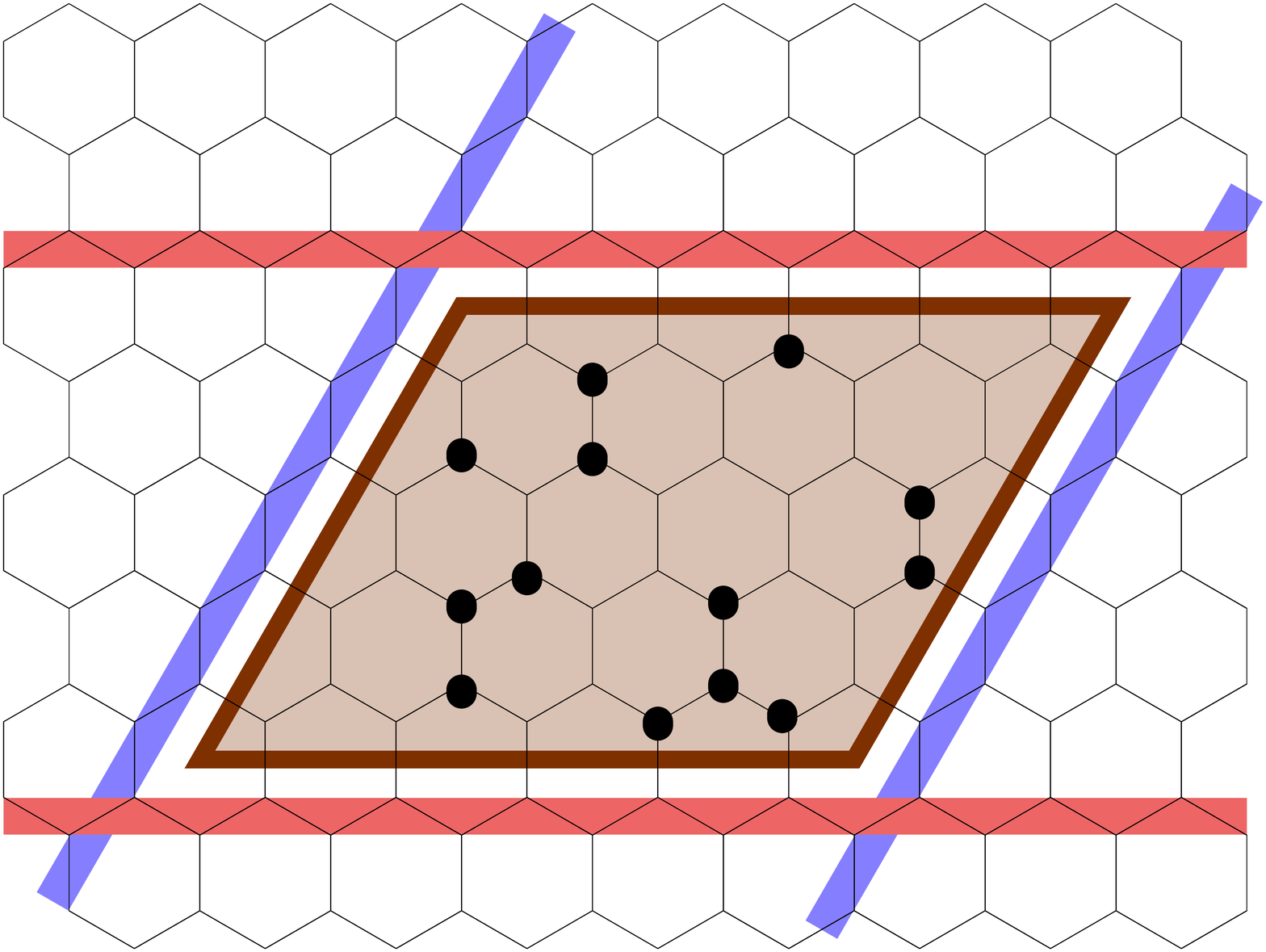}}\hfill
\subfloat[Eliminating a bad row of direction $3$ by moving vertices of component $B_2$ agreeable to direction $2$]
{\label{fig:parallelwhite_b}\includegraphics[totalheight=3.4cm]{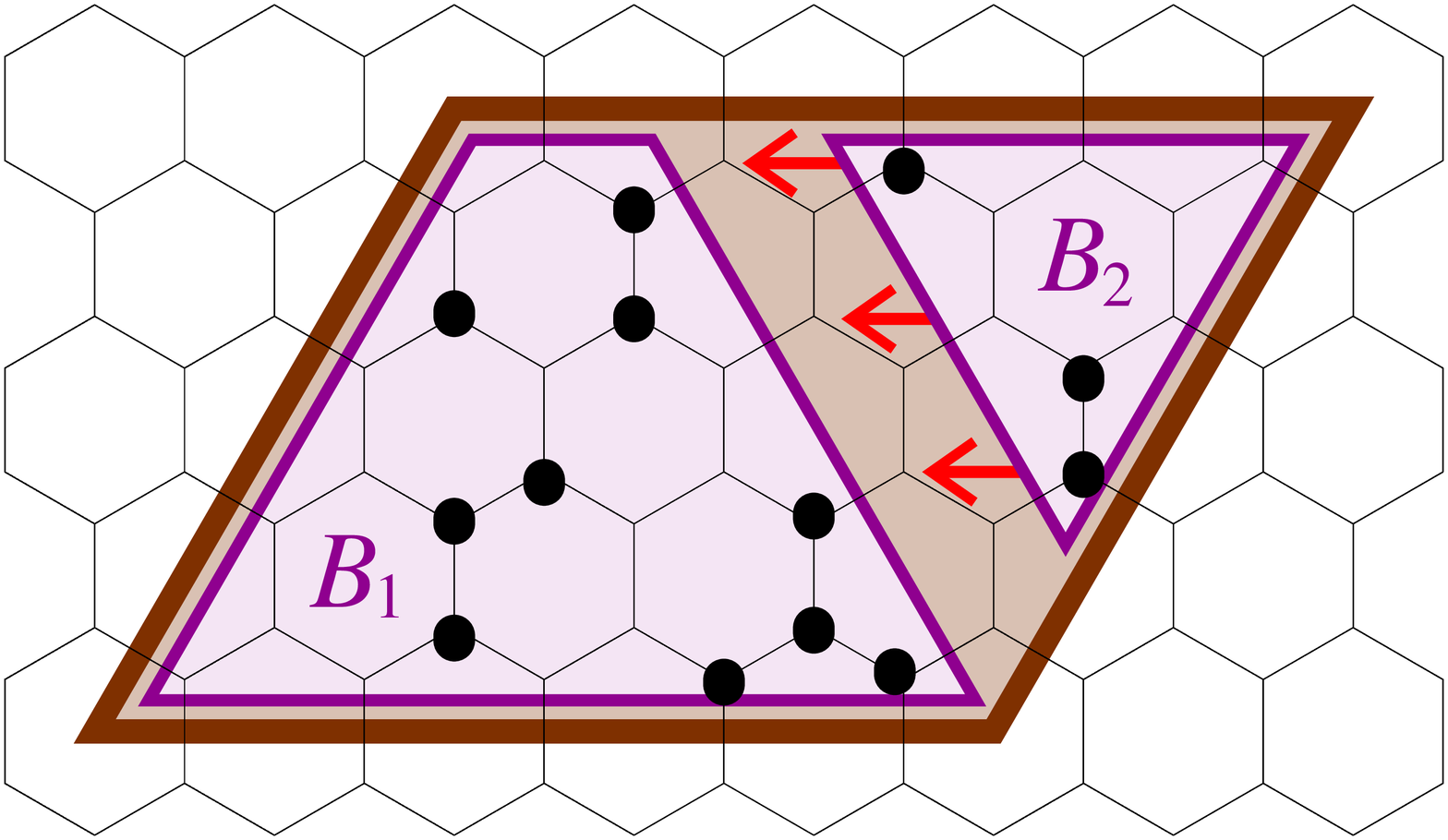}}
\end{minipage}
\caption{Eliminating bad rows}
\label{fig:parallelwhite}
\end{figure}

We can repeat this procedure until we get a set $W'$ with no bad rows of direction $1$ or $2$.
It is easy to see that $W'$ has the same cardinality
and at most as many neighbor vertices as $W\!$, and 
as we removed all bad rows of direction $3$ in the last step, the set $W'$ has no bad rows.
\end{proof}

%

In the following we assume that $W$ is such that there are no bad rows
for all directions. 

\paragraph{Part 2 \textnormal{(Lower Bound for $|\CN(W)|$)}.}
Let $l_i$ be the number of gray rows in direction $i$. 
Without loss of generality, let $l_3\geq\max\{l_1,l_2\}$.
For any direction there are two outermost neighbors of $W$ within each gray row. Thus, 
direction $i$ identifies $2 l_i$ outermost neighbors.
In Figure~\ref{fig:outermostneighbors_a} all outermost neighbors regarding direction $2$ are shown.
Of course, for each direction the outermost neighbors of the gray rows are neighbor vertices, and therefore, belong to $\CN(W)$.
Hence, for each direction the number of outermost neighbors is a lower bound on $|\CN(W)|$.
In order to increase this lower bound, we consider all three directions.
If we take a look at the outermost neighbors for direction~$i$, we observe that some
of these neighbors might also occur as outermost neighbors for other directions 
(see Figure~\ref{fig:outermostneighbors_b}).

\begin{figure}[ht]
\centering
\begin{minipage}{0.65\textwidth}
\subfloat[Outermost neighbors for direction $2$]
{\label{fig:outermostneighbors_a}\resizebox{!}{3.8cm}{\input{fig6neighborschrift.tex}}}\hfill
\subfloat[Outermost neighbors for direction~$2$ and rows of direction~$1$; 
the encircled neighbors are also outermost neighbors for direction~$1$]
{\label{fig:outermostneighbors_b}\includegraphics[totalheight=3.8cm]{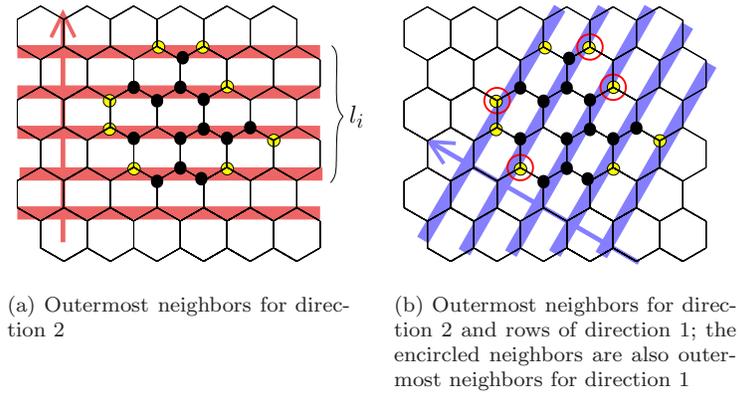}}
\end{minipage}
\caption{Outermost neighbors}
\label{fig:outermostneighbors}
\end{figure}

\begin{lemma}\label{lem:neighbortwice}
Every neighbor is an outermost neighbor for at most two directions.
\end{lemma}
\begin{proof}
Assume vertex $v$ is outermost neighbor for three directions.
Let $v_{1}$,  $v_{2}$ and $v_{3}$ be the neighbors of $v$,
such that regarding direction $1$ the vertices $v_2$ and $v_3$,
for direction $2$ the vertices $v_1$ and $v_3$
and for direction $3$ the vertices $v_2$ and $v_3$ occur in the same row as $v$ 
(see Figure~\ref{fig:neighborthreea}).
As $v$ is an outermost neighbor regarding direction $1$, either $v_2$ is a black vertex and $v_3$ is white or 
the other way round.
As $G_{\infty}$ is symmetric, we can assume $v_2$ is a black and $v_3$ is a white vertex 
(see Figure~\ref{fig:neighborthreeb}).
Consequently, $v_1$ has to be a black vertex, because  $v_3$ is white and $v$ is an outermost neighbor 
for direction $2$ (see Figure~\ref{fig:neighborthreec}).
But then $v$ is no outermost neighbor for 
direction $3$ (see Figure~\ref{fig:neighborthreed}), a contradiction.
\end{proof}

\begin{figure}[ht]
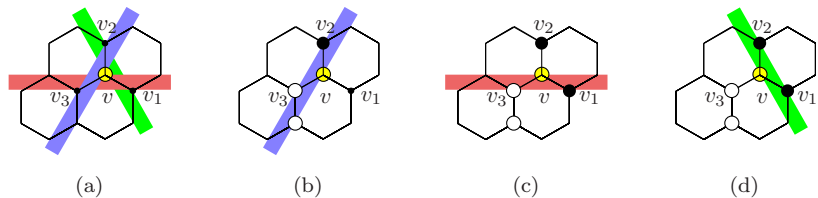

\centering
\begin{minipage}{0.75\textwidth}
\subfloat[]{\label{fig:neighborthreea}\resizebox{!}{2.2cm}{\input{fig7neighbthree1.tex}}}\hfill
\subfloat[]{\label{fig:neighborthreeb}\resizebox{!}{2.2cm}{\input{fig7neighbthree2.tex}}}\hfill
\subfloat[]{\label{fig:neighborthreec}\resizebox{!}{2.2cm}{\input{fig7neighbthree3.tex}}}\hfill
\subfloat[]{\label{fig:neighborthreed}\resizebox{!}{2.2cm}{\input{fig7neighbthree4.tex}}}
\end{minipage}
\caption{Counting directions for which vertex $v$ is outermost neighbor}
\label{fig:neighborthree}
\end{figure}

Applying Lemma~\ref{lem:neighbortwice}, we obtain the following lower bound for the number 
of vertices in $\CN(W)$:  
\begin{equation}\label{equ:neighborvertices}
  |\CN(W)|\geq\frac12\cdot(2l_1+2l_2+2l_3)=l_1+l_2+l_3.
\end{equation}

\paragraph{Part 3 \textnormal{(Upper Bound for $|W|$)}.}
Now, we want to determine an upper bound for the number of vertices in $W$ in terms of $l_1$, $l_2$ and $l_3$.
For example $2l_1l_2$ is a possible upper bound on the number of vertices in $W$ (see Figure~\ref{fig:parallelogramma}).
It is the number of vertices within a parallelogram with $l_1$ rows in direction 1 and $l_2$ rows in direction 2.
(Remember that there are no white rows in between gray rows for direction $1$ and $2$).
But this bound does not consider $l_3$. Knowing the number of gray rows regarding the third direction, 
we can exclude some third direction rows of the parallelogram (see Figure~\ref{fig:parallelogrammb}).
Since the outermost rows contain the least vertices, we exclude them for an upper bound,
and as 
the number of rows of direction 3 in the parallelogram are $l_1+l_2$, 
we exclude $l_1+l_2-l_3$ outermost rows.

\begin{figure}[ht]
\centering
\begin{minipage}{0.75\textwidth}
\subfloat[]{\label{fig:parallelogramma}\includegraphics[totalheight=3.3cm]{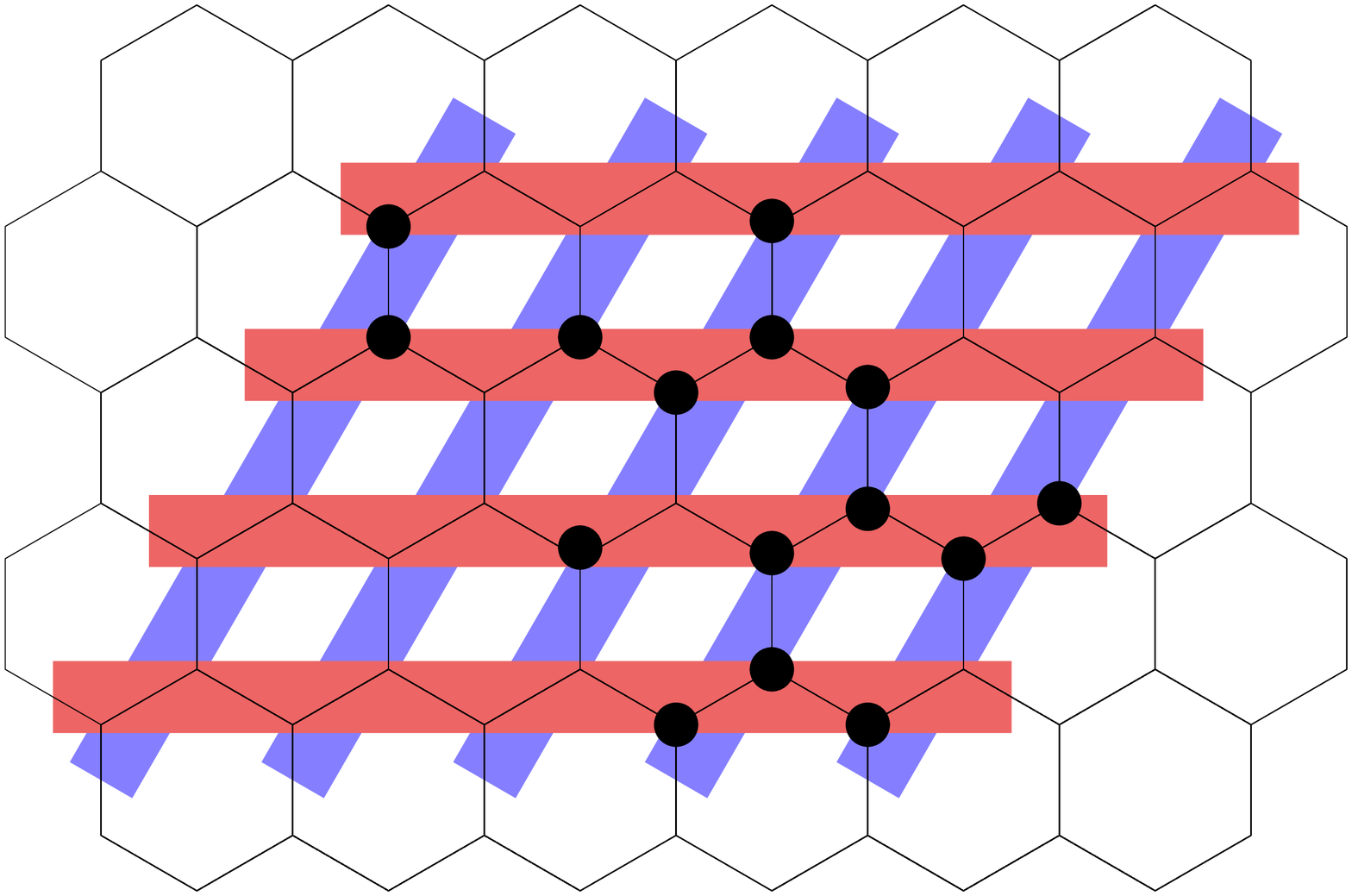}}\hfill
\subfloat[]{\label{fig:parallelogrammb}\includegraphics[totalheight=3.3cm]{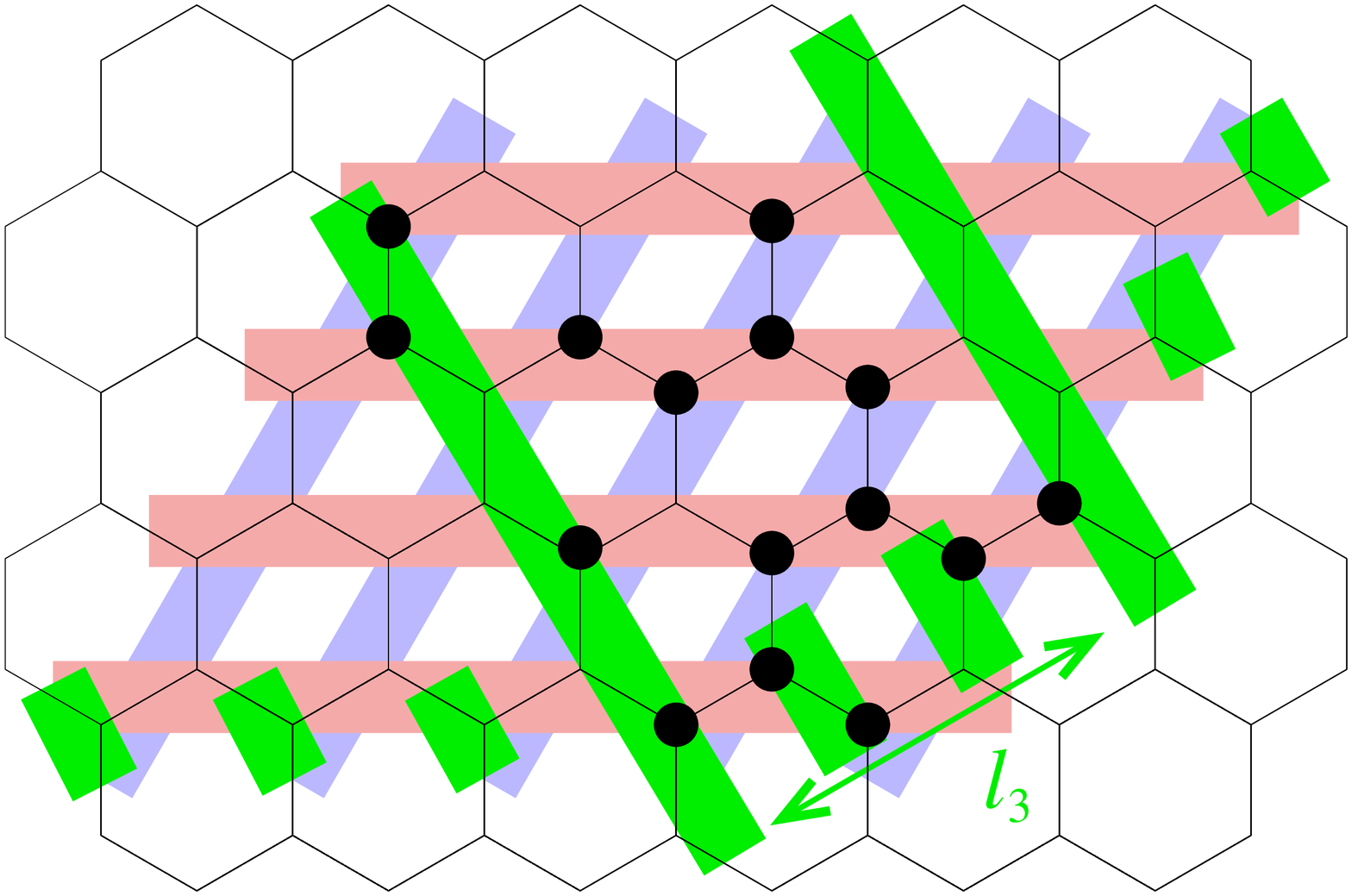}}
\end{minipage}
\caption{Upper bound on the number of vertices in $W$ using a parallelogram}
\label{fig:parallelogramm}
\end{figure}
\begin{figure}[ht]
\centering
\includegraphics[totalheight=3.3cm]{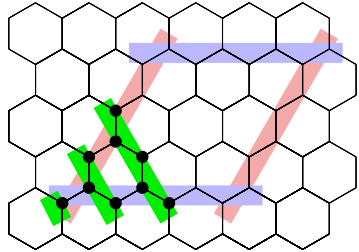}
\caption{Counting vertices inside parallelogram}
\label{fig:countingvertices}
\end{figure}

We know that $l_1+l_2-l_3\leq \min\{l_1,l_2\}$, because $l_3\geq\max\{l_1,l_2\}$. 
Thus,  
if we exclude $k$ rows from one side, then  
$k\leq \min\{l_1,l_2\}$ and we remove $\sum_{i=1}^{k}(2i-1)=k^2$ vertices (see Figure~\ref{fig:countingvertices}).
Now we have to distinguish between two cases:
If  $l_1+l_2-l_3$  is even, we can exclude $\frac{1}{2}(l_1+l_2-l_3)$ rows from each side, and therefore, 
we can remove $2(\frac{1}{2}(l_1+l_2-l_3))^2$ vertices from the parallelogram.
If $l_1+l_2-l_3$  is odd, we achieve the best result by excluding $\frac{1}{2}(l_1+l_2-l_3-1)$ rows from one side and 
$\frac{1}{2}(l_1+l_2-l_3+1)$ rows from the other side. 
Hence, we can bound the number of vertices that we can remove by
\begin{align*}
 &\left(\frac{1}{2}(l_1+l_2-l_3-1)\right)^2+\left(\frac{1}{2}(l_1+l_2-l_3+1)\right)^2\\
=\quad&\ \frac{1}{2}(l_1+l_2-l_3)^2+\frac12.
\end{align*}

Therefore we can remove at least $\lceil\frac12\left(l_1+l_2-l_3\right)^2\rceil$ vertices from the parallelogram,
and can give a better upper bound on the number of vertices in $W$:
\begin{align}
\nonumber    |W|\quad\leq\quad &   2l_1l_2 - \frac12(l_1+l_2-l_3)^2\\
\label{equ:cardinality}   = \quad& - \frac12 (l_1^2+ l_2^2 + l_3^2) + (l_1l_2 + l_1l_3 +l_2l_3).
%
\end{align}
\paragraph{Part 4 \textnormal{(Proof)}.}
Finally, we can proof Theorem~\ref{thm:infinitegrid}.\ref{lab:thm:infinitegrid1}.
\begin{proof}[Proof of Theorem~\ref{thm:infinitegrid}.\ref{lab:thm:infinitegrid1}]
We use the lower bound on $|\CN(W)|$ and the upper bound on $|W|$ to show that the inequality of  
Theorem~\ref{thm:infinitegrid}.\ref{lab:thm:infinitegrid1} holds:
\begin{align*}
	|\CN(W)|^2\quad\geq\quad&  (l_1+l_2+l_3)^2\\
	                           \geq\quad&  6\cdot \big(-\frac12 (l_1^2+ l_2^2 + l_3^2) + (l_1l_2 + l_1l_3 +l_2l_3)\big)\\
	                           \geq\quad&  6 \cdot |W|.
\end{align*}
Here, the first and third inequality are shown in (\ref{equ:neighborvertices}) and (\ref{equ:cardinality}), respectively.
The second inequality can be verified by an easy calculation.
In order to show that the inequality is tight, we give an example.
For an arbitrary $r\in \N$, let us consider the subset $W:=V(G_r)$ of vertices of the infinite grid $G_{\infty}$.
Then the tightness of the inequality follows directly from Lemma~\ref{lem:NbOfVerticesOfFinHexGrid}. 
\end{proof}

The proof of  Theorem~\ref{thm:infinitegrid}.\ref{lab:thm:infinitegrid2} 
works analogous to  the one of Theorem~\ref{thm:infinitegrid}.\ref{lab:thm:infinitegrid1}: 
\begin{proof}[Proof of Theorem~\ref{thm:infinitegrid}.\ref{lab:thm:infinitegrid2}]
If we use outermost outgoing egdes instead of  outermost neighbor vertices, 
we can derive  the same lower bound for the number of outgoing edges as for the number of neighbor vertices.
Then, the rest of the proof of the inequality in Theorem~\ref{thm:infinitegrid}.\ref{lab:thm:infinitegrid2} 
works as in Theorem~\ref{thm:infinitegrid}.\ref{lab:thm:infinitegrid1}.
Since $|\CN(V(G_r))|=|\CE(V(G_r))|$ (as shown in Lemma~\ref{lem:NbOfVerticesOfFinHexGrid}),
the set $W:=V(G_r)$ is also an example for the inequality of    Theorem~\ref{thm:infinitegrid}.\ref{lab:thm:infinitegrid2} 
being tight.
%
\end{proof}

Finally, we proof Theorem~\ref{thm:infinitegrid}.\ref{lab:thm:infinitegrid3} using 
Theorem~\ref{thm:infinitegrid}.\ref{lab:thm:infinitegrid1}.

\begin{proof}[Proof of Theorem~\ref{thm:infinitegrid}.\ref{lab:thm:infinitegrid3}]
Let us consider a finite subset $W\subset G_\infty$.
Let $X$ be the set $W\setminus\CB(W)$.
Then, the neighbor vertices $\CN(X)$ are a subset of $\CB(W)$.
Therefore, we have $|\CN(X)|\leq|\CB(W)|$.
Using this inequality and Theorem~\ref{thm:infinitegrid}.\ref{lab:thm:infinitegrid1} for $c=\sqrt{6}$ we obtain 
\begin{equation}\label{equ:boundary}
 |\CB(W)|\geq|\CN(X)|\geq c\cdot \sqrt{|X|}= c\cdot\sqrt{|W|-|\CB(W)|}.
\end{equation}
An easy calculation shows the inequality in Theorem~\ref{thm:infinitegrid}.\ref{lab:thm:infinitegrid3}:
\begin{align*}
        &&    |\CB(W)|\ \quad\geq\quad& c\cdot\sqrt{|W|-|\CB(W)|}\quad\quad\quad\quad\quad\quad\quad\quad\\
\iff\quad&&    |\CB(W)|^2\quad\geq\quad& c^2\cdot(|W|-|\CB(W)|)\\
\iff\quad&&    (|\CB(W)|+\frac{c^2}2)^2\quad\geq\quad& c^2|W|+\frac{c^4}4\\
\iff\quad&&    |\CB(W)|+\frac{c^2}2\quad\geq\quad& c\cdot \sqrt{|W|+\frac{c^2}4}.
\end{align*}
Again, we give an example to show that the inequality is tight:
For an arbitrary $r\in \N$, let us consider the subset $W:=V(G_r)\cup \CN(V(G_r))$ of vertices of the infinite grid $G_{\infty}$.
An illustration of this subset can be found in Figure~\ref{fig:exampleboundary}.
Since $|\CN(X)|=|\CB(W)|$ in this example, 
and the inequality of Theorem~\ref{thm:infinitegrid}.\ref{lab:thm:infinitegrid1} is tight for the subset $X=V(G_r)$,
the inequality in~(\ref{equ:boundary}) is tight, 
and therefore also the inequality in Theorem~\ref{thm:infinitegrid}.\ref{lab:thm:infinitegrid3}.

\end{proof}
\begin{figure}[h]
\centering
\includegraphics[width=0.2\textwidth]{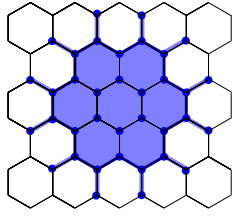}
\caption{The set $V(G_2)\cup \CN(V(G_2))$}
\label{fig:exampleboundary}
\end{figure}

\section{Finite Hexagonal Grid}


For finite hexagonal grids, we obtain the following results:

\begin{theorem}\label{thm:finite}
	For all $r\geq 0$ and all finite subsets $W\subseteq V(G_r)$ with $|W|\leq \frac{1}2 |V(G_r)|$ we have
	\begin{enumerate}
		\item\label{thm:finite1} $|\CN(W)|\geq c\cdot \sqrt{|W|}$ for $c=0.6053$. 
		\item\label{thm:finite2} $|\CE(W)|\geq c\cdot \sqrt{|W|}$ for $c=\sqrt{6}-\sqrt{3}\approx 0.7174$. 
		\item\label{thm:finite3} $|\CB(W)|\geq c\cdot \sqrt{|W|+\frac14 c^2}-\frac12 c^2$ for $c=0.6053$. 
	\end{enumerate}
	Furthermore, for all $c<\sqrt{6}-\sqrt{3}\approx 0.7174$ there exists an $R_c\in \N$ such that for all
	$r\geq R_c$ and all $W\subset V(G_r)$ with  $|W|\leq \frac{1}2 |V(G_r)|$, we have
	\begin{enumerate}\setcounter{enumi}{3}
		\item\label{thm:finite4}  $|\CN(W)|\geq c\cdot \sqrt{|W|}$ 
		\item\label{thm:finite5}  $|\CB(W)|\geq c\cdot \sqrt{|W|+\frac14 c^2}-\frac12 c^2$.
	\end{enumerate}
\end{theorem}

\begin{proof}[Proof of Theorem~\ref{thm:finite}.\ref{thm:finite1} and~\ref{thm:finite}.\ref{thm:finite4}]
Let $G_\infty$ be the infinite hexagonal grid and the subgraph $G_r$ be the finite hexagonal grid of radius~$r\geq0$. 
For a subset $Z\subseteq V(G_r)$ let us denote the set of neighbors of $Z$ 
that lie inside $G_r$ with $\CN_{in}(Z)$, and all neighbors of $Z$ 
outside of $G_r$ with $\CN_{out}(Z)$.
Let $W$ be a subset of $V(G_r)$ with $|W|\leq\frac12|V(G_r)|$,
and let $r\geq 2$, as for $r\in\{0,1\}$ Theorem~\ref{thm:finite}.\ref{thm:finite1} trivially holds.
In order to show Theorem~\ref{thm:finite}.\ref{thm:finite1}, we assume
that $\CN_{in}(W)$ is small, that is,
\begin{equation}
\label{equ:1assumption}
 |\CN_{in}(W)|<c\cdot \sqrt{|W|},
\end{equation}
 for $c=0.6053$, and deduce a contradiction by exploiting the results in 
Theorem~\ref{thm:infinitegrid}.\ref{lab:thm:infinitegrid1} and 
Lemma~\ref{lem:NbOfVerticesOfFinHexGrid}.

We define $U:=V(G_r)\setminus(W\cup\CN_{in}(W))$ (see Figure~\ref{fig:finitea}), and by Lemma~\ref{lem:NbOfVerticesOfFinHexGrid}
and (\ref{equ:1assumption}) we obtain
a lower bound for the number of vertices in $U$:
\begin{eqnarray}
\nonumber|U|\quad =&|(V(G_r)|-|W|-|N_{in}(W)|\\
\label{eqn:2U} >&\quad\quad\ 6r^2 -|W|-c\cdot \sqrt{|W|}.
\end{eqnarray}

Let us now consider  $\CN_{out}(W)$, the set of neighbors of $W$ that lie outside of $G_r$ (see Figure~\ref{fig:finiteb}).
By (\ref{equ:1assumption}), the assumption that $\CN_{in}(W)$ is small, and Theorem~\ref{thm:infinitegrid}.\ref{lab:thm:infinitegrid1} 
we obtain that  $\CN_{out}(W)$ is quite large:
\begin{align}
\nonumber|N_{out}(W)|\quad=&\quad\ \ |N(W)|-|N_{in}(W)|\\
\label{eqn:3outneighW} >&\quad\sqrt{6}\sqrt{|W|}-c\sqrt{|W|}.
\end{align}

Next, we consider $\CN_{out}(U)$, the neighbors of $U$ that lay outside of $G_r$ (see Figure~\ref{fig:finitec}).
As $U$ and $W$ are disjoint and there are no neighbors of $V(G_r)$ that are connected to more 
than one vertex of the grid $G_r$, the sets  $\CN_{out}(U)$ and $\CN_{out}(W)$ are disjoint.
Thus, it follows that $N_{out}(U)$ is rather small:
\begin{align}
\nonumber|N_{out}(U)|\quad\leq&\quad |N(G_r)|-|N_{out}(W)|\\
\label{eqn:4outneighU}<&\quad\quad\quad\  6r-(\sqrt{6}-c)\sqrt{|W|},
\end{align}
where the last inequality follows from Lemma~\ref{lem:NbOfVerticesOfFinHexGrid} and (\ref{eqn:3outneighW}).

Finally, we use Theorem~\ref{thm:infinitegrid}.\ref{lab:thm:infinitegrid1} and (\ref{eqn:4outneighU}) 
to show that $N_{in}(U)$, the set of neighbors of $U$ that lie inside $G_r$ (see Figure~\ref{fig:finited}),
is again rather large,
and we use~(\ref{eqn:2U}) to obtain a bound on $|N_{in}(U)|$ that depends only on $c$, $r$ and $|W|$:
\begin{align}
\nonumber  |N_{in}(U)|\quad =&\quad\ \ |N(U)|-|N_{out}(U)|\\
\label{eqn:5ainneighU}  \geq&\quad \sqrt{6}\sqrt{|U|}-6r+(\sqrt{6}-c)\sqrt{|W|}\\
\label{eqn:5inneighU}  \geq&\quad \sqrt{6}\sqrt{6r^2-|W|-c\cdot \sqrt{|W|}}-6r+(\sqrt{6}-c)\sqrt{|W|}.
\end{align}
Since $|W|\leq 3r^2$, we note that  $6r^2-|W|-c\cdot \sqrt{|W|}\geq 0$ for $c\leq \frac{6}{\sqrt{3}}\approx 3.4641$.

\vspace{0.3cm}
\begin{figure}[ht]
\centering
\begin{minipage}{1\textwidth}
\subfloat[]{\label{fig:finitea}\includegraphics[totalheight=3.1cm]{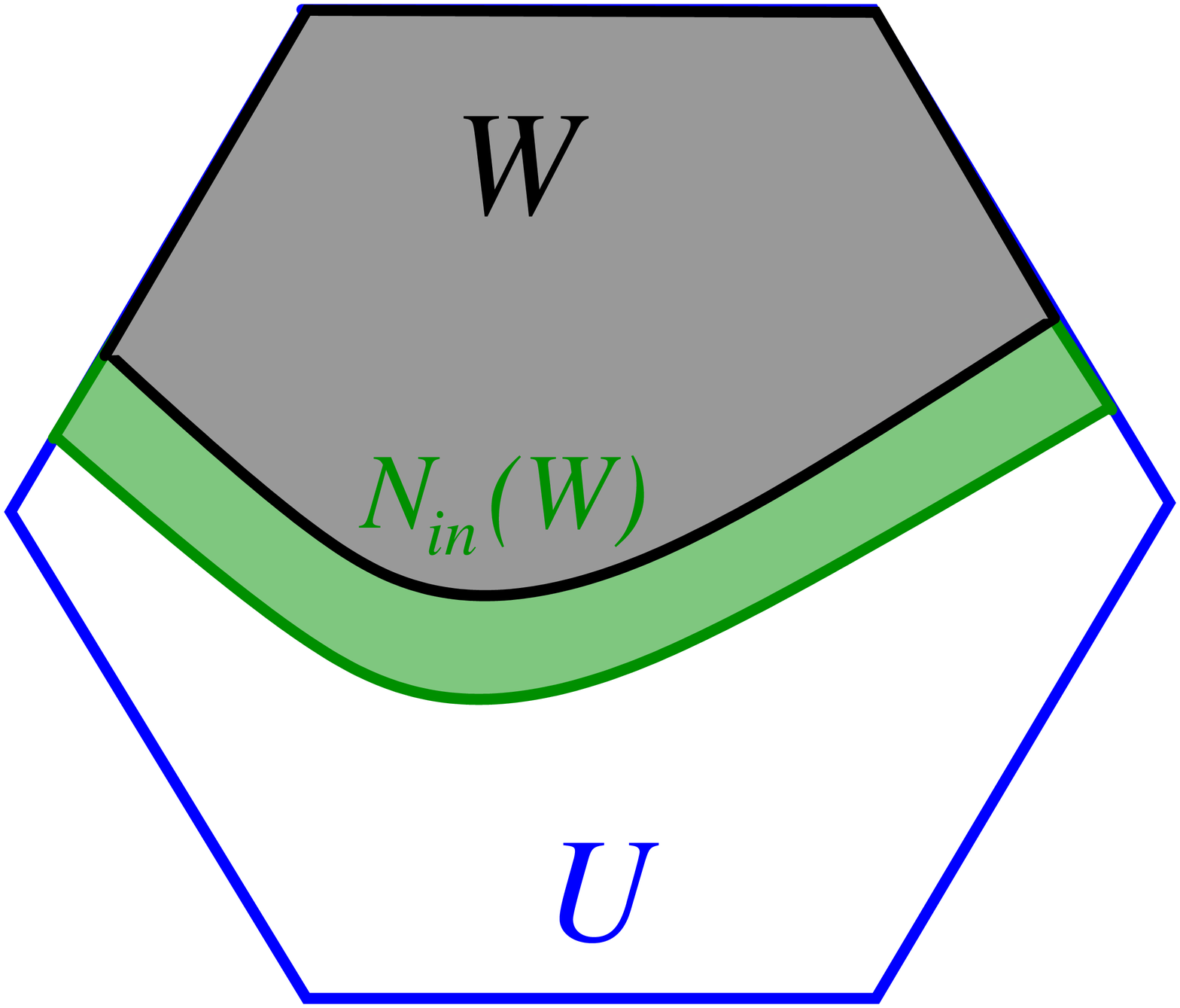}}\hfill
\subfloat[]{\label{fig:finiteb}\includegraphics[totalheight=3.1cm]{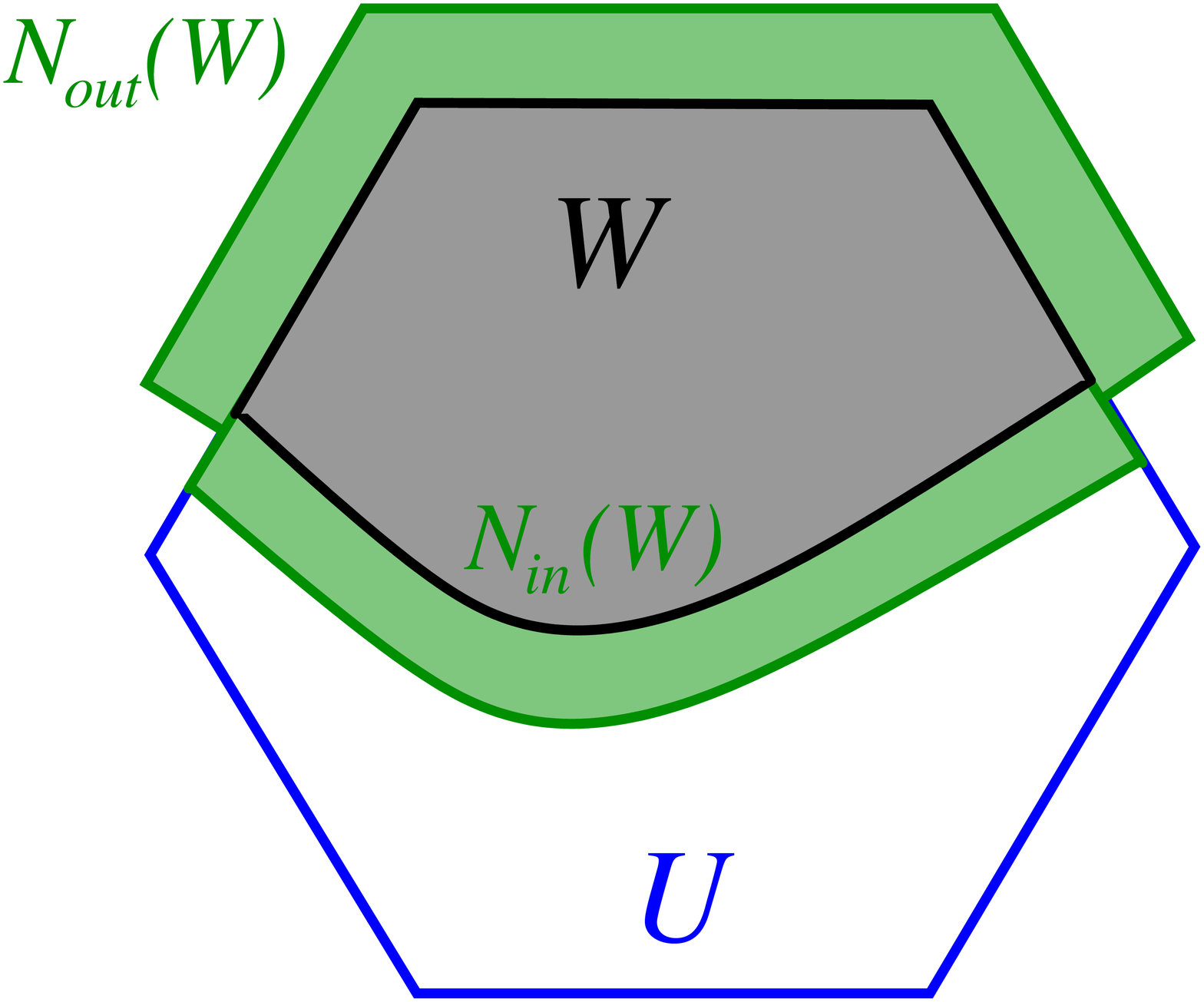}}\hfill
\subfloat[]{\label{fig:finitec}\includegraphics[totalheight=3.1cm]{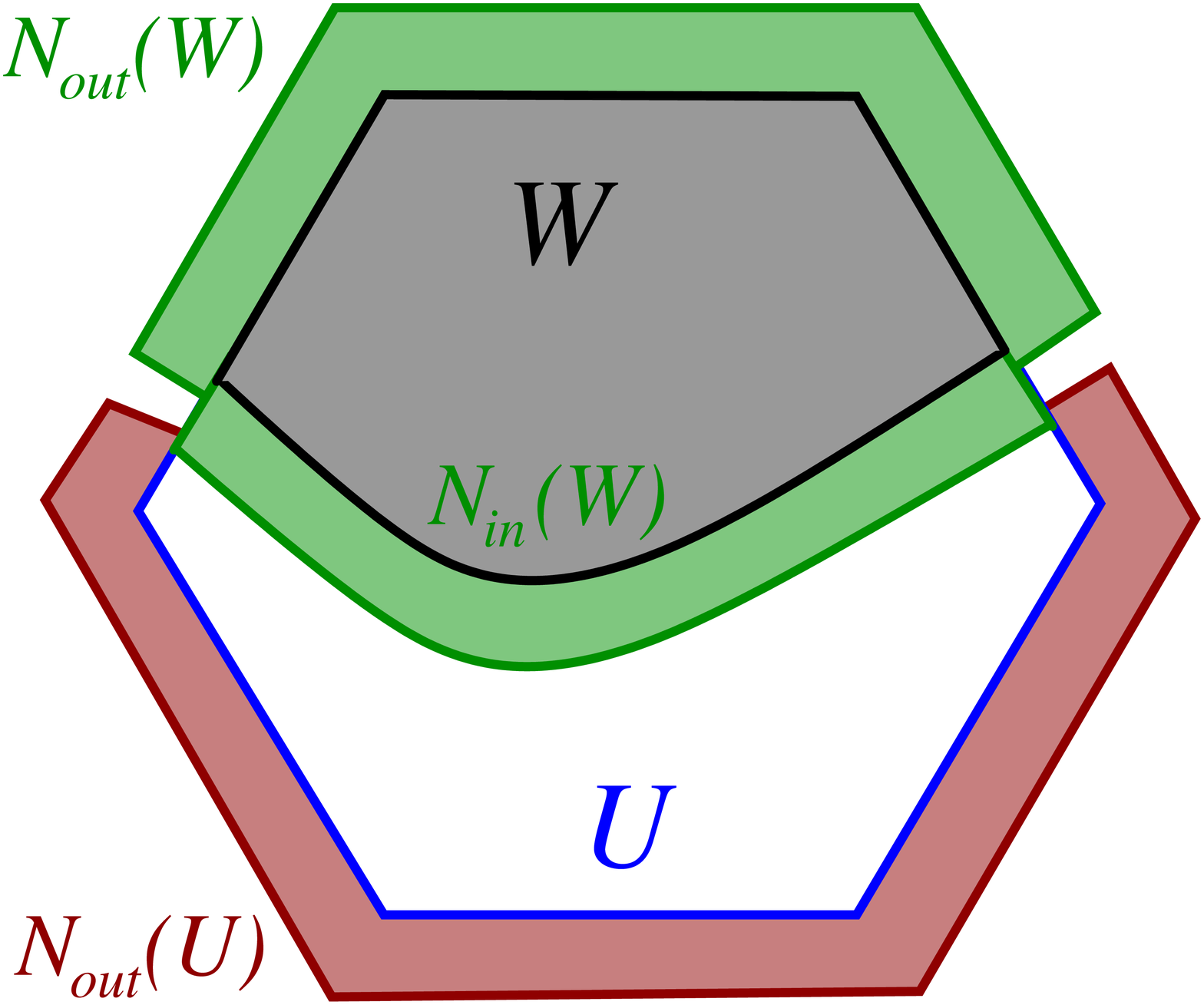}}\hfill
\subfloat[]{\label{fig:finited}\includegraphics[totalheight=3.1cm]{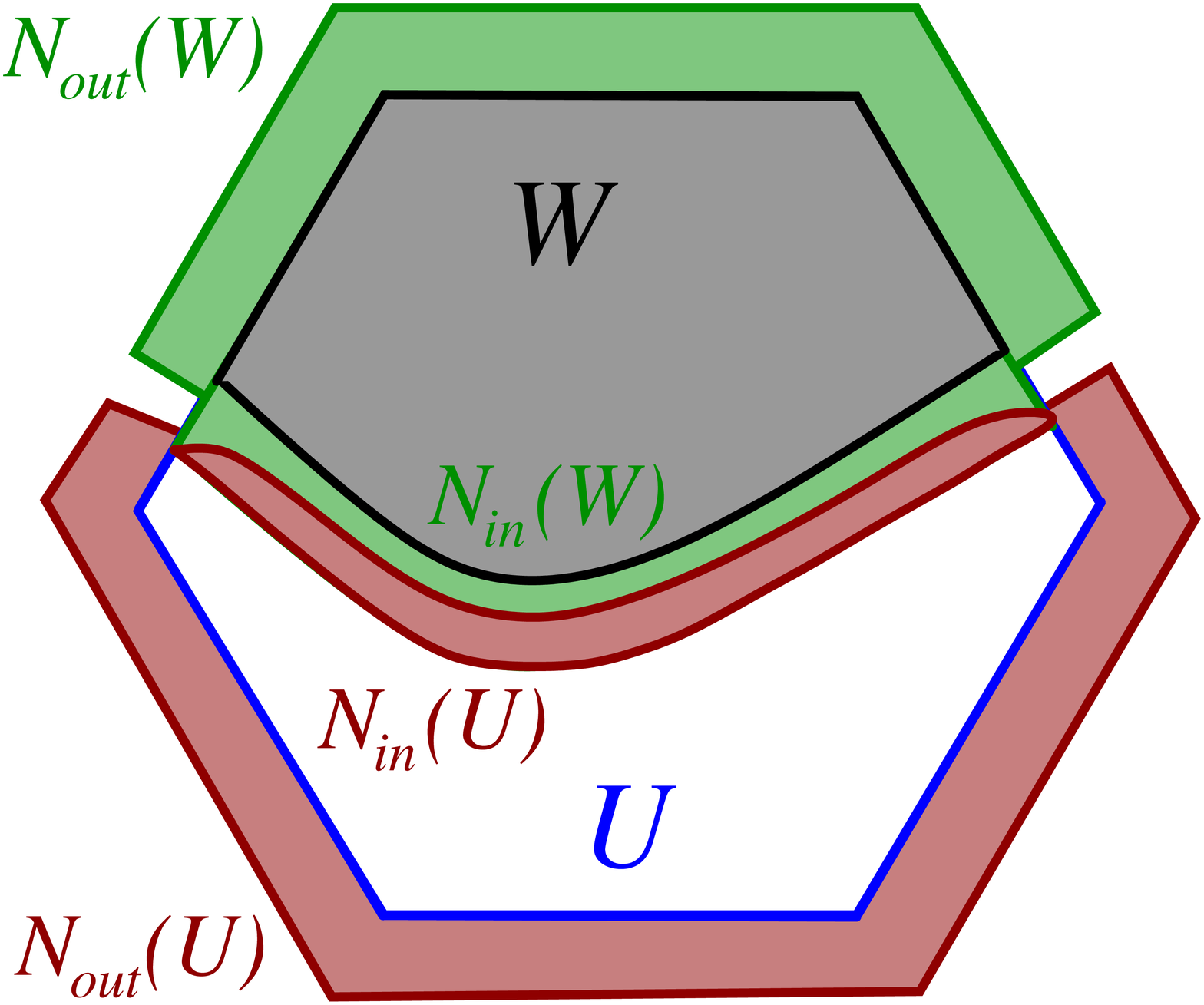}}
\end{minipage}
\caption{Vertex neighbors in the finite grid $G_r$}
\label{fig:finitegridproof}
\end{figure}

It is easy to see that every neighbor of $U$ belongs to $\CN_{in}(W)$. 
Thus, $\CN_{in}(U)\subseteq \CN_{in}(W)$, and we can derive the following inequality using~(\ref{equ:1assumption}) and~(\ref{eqn:5inneighU}): 
\begin{align}
\nonumber c\cdot \sqrt{|W|}> |N_{in}(W)|  \geq&\ |N_{in}(U)|\\
\label{eqn:6inequality}\geq&\ \sqrt{6}\cdot \sqrt{6r^2-|W|-c\cdot \sqrt{|W|}}-6r+(\sqrt{6}-c)\sqrt{|W|}   
\end{align}

We can transform the inequality in~(\ref{eqn:6inequality}) the following way:
\begin{align*}
&&    c\cdot \sqrt{|W|}\quad >&\quad\sqrt{6}\cdot \sqrt{6r^2-|W|-c\cdot \sqrt{|W|}}-6r+(\sqrt{6}-c)\sqrt{|W|} \\
\Longrightarrow&& 
(2c-\sqrt{6})\cdot \sqrt{|W|}+6r\quad >&\quad\sqrt{6}\cdot \sqrt{6r^2-|W|-c\cdot \sqrt{|W|}} 
\end{align*}

Since the right side of the inequality is at least 0, we can sqare both sides of the inequality:
\begin{align}
\nonumber\Longrightarrow&& 
(2c-\sqrt{6})^2\cdot |W|+12(2c-\sqrt{6})r\sqrt{|W|}+ 36r^2\quad >&\quad 6\cdot (6r^2-|W|-c\cdot \sqrt{|W|}) \\
\nonumber\Longrightarrow&& 
(4c^2-4\sqrt{6}c+12)\cdot  \sqrt{|W|}\quad >&\quad -6c-(24c-12\sqrt{6})r \\
\nonumber\Longrightarrow&&    \frac{-4c^2+4\sqrt{6}c-12}{24c-12\sqrt{6}}\sqrt{|W|}\quad>&\quad\frac{6c}{24c-12\sqrt{6}}+r\\
\intertext{The last implication holds because $-(24c-12\sqrt{6})> 0 \iff c<\frac12 \sqrt{6}\approx 1.2247$.\newline
As $(-4c^2+4\sqrt{6}c-12)<0$ for all $c\in \R$, $(24c-12\sqrt{6})< 0$ for $c<\frac12 \sqrt{6}$ and $|W|\leq3r^2$, we obtain}
\nonumber\Longrightarrow&&    \frac{-4c^2+4\sqrt{6}c-12}{24c-12\sqrt{6}}\sqrt{3r^2}\quad>&\quad\frac{6c}{24c-12\sqrt{6}}+r\\
\nonumber\Longrightarrow&&    \frac{-4\sqrt{3}c^2+12\sqrt{2}c-12\sqrt{3}-24c+12\sqrt{6}}{24c-12\sqrt{6}}r\quad>&\quad\frac{6c}{24c-12\sqrt{6}}\\
\intertext{Since  $(24c-12\sqrt{6})< 0$ for $c<\frac12 \sqrt{6}$:}
\nonumber\Longrightarrow&&    (-4\sqrt{3}c^2+12\sqrt{2}c-12\sqrt{3}-24c+12\sqrt{6})\ r\quad<&\quad 6c\\
\nonumber\Longrightarrow&&    -4\sqrt{3}\left(c+\sqrt{3}\right)\left(c-\sqrt{6}+\sqrt{3}\right)\ r\quad<&\quad 6c\\
\intertext{Thus, for $-\sqrt{3}<c<\sqrt{6}-\sqrt{3}$, it follows that}
\label{equ:functionf}\Longrightarrow&&    r\quad<\quad \frac{\sqrt{3}c}{-2\sqrt{3}\left(c+\sqrt{3}\right)\left(c-\sqrt{6}+\sqrt{3}\right)}.&
\end{align}
Figure~\ref{fig:c-r} depicts the graph of the function $f(c)= \frac{\sqrt{3}c}{-2\sqrt{3}\left(c+\sqrt{3}\right)\left(c-\sqrt{6}+\sqrt{3}\right)}$.

For $c=0.6053$, we have $f(c)\leq2$. Therefore, the inequality in (\ref{equ:functionf}) is not satisfied, as we assumed $r\geq 2$.
Hence, we obtain a contradiction.

Further,  for all $0\leq c<\sqrt{6}-\sqrt{3}\approx 0.7174$ we can set $R_c:=f(c)$ and  conclude that 
for all $r\geq \lceil R_c\rceil$ and all $W\subset V(G_r)$ with $|W|\leq 3r^2$ we have
$|\CN(W)|\geq c\cdot \sqrt{|W|}$. 
%
\begin{figure}[ht]
\centering
\begin{minipage}{0.75\textwidth}\centering
\includegraphics[totalheight=6.0cm]{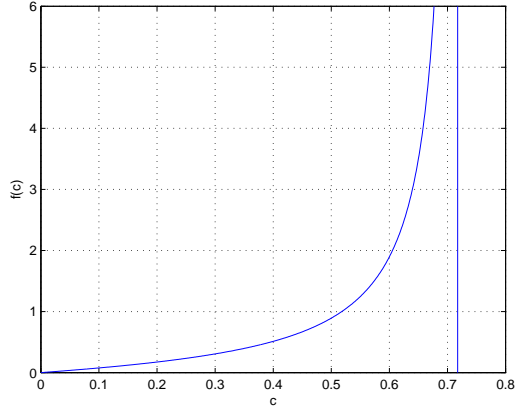}
\caption{The graph of function $f(c)= \frac{\sqrt{3}c}{-2\sqrt{3}\left(c+\sqrt{3}\right)\left(c-\sqrt{6}+\sqrt{3}\right)}$,
which gives for each ${0\leq c<\sqrt{6}-\sqrt{3}}$ a lower bound $R_c$ on the radius}
\label{fig:c-r}
\end{minipage}
\end{figure}
\end{proof}
\begin{proof}[Proof of Theorem~\ref{thm:finite}.\ref{thm:finite2}]
In principle the isoperimetric inequality in  Theorem~\ref{thm:finite}.\ref{thm:finite2} can be derived as the one in 
  Theorem~\ref{thm:finite}.\ref{thm:finite1}.
Again, let $G_\infty$ be the infinite hexagonal grid and the subgraph $G_r$ be the finite hexagonal grid with radius~$r\geq0$. 
For a set $Z$ of vertices, we denote the set of outgoing edges of $Z$ that lie inside $G_r$ with $\CE_{in}(Z)$,
the set of outgoing edges of $Z$ that are incident to a vertex outside of $G_r$ with $\CE_{out}(Z)$, 
and for $c=\sqrt{6}-\sqrt{3}$ we assume
\begin{equation}
\label{equ:E1assumption}
 |\CE(W)|<c\cdot \sqrt{|W|}.
\end{equation}
This time, we define $U:=V(G_r)\setminus W$.
Then
\begin{align}
\nonumber|U|\quad =&\quad|(V(G_r)|-|W|\\
\nonumber>&\quad\quad\quad\ 6r^2 -|W|.
\end{align}
For (\ref{eqn:3outneighW})--(\ref{eqn:5ainneighU}) we get analogous inequalities regarding the number of outgoing edges,
and as $\CE_{in}(W)=\CE_{in}(U)$ we obtain
\begin{align}
\nonumber&& c\cdot \sqrt{|W|}\quad >&\quad \sqrt{6}\cdot \sqrt{6r^2-|W|}-6r+(\sqrt{6}-c)\sqrt{|W|} \\ 
\nonumber\Longrightarrow&& (2c-\sqrt{6})\cdot \sqrt{|W|} +6r\quad >&\quad \sqrt{36r^2-6|W|} \\ 
\nonumber\Longrightarrow&& (2c-\sqrt{6})^2|W|+12(2c-\sqrt{6})r \sqrt{|W|} +36r^2\quad >&\quad 36r^2-6|W| \\ 
\nonumber\Longrightarrow&& (24c-12\sqrt{6})r  \quad >&\quad   (-4c^2+4c\sqrt{6}-12) \sqrt{|W|}\\ 
\intertext{As $(-4c^2+4c\sqrt{6}-12)<0$ for all $c$ and $|W|\leq 3r^2$, we get}
\nonumber\Longrightarrow&& (24c-12\sqrt{6})r \quad >&\quad   (-4c^2+4c\sqrt{6}-12)\sqrt{3}r \\
\label{eqn:th22inequal}\Longrightarrow&& \frac{(24c-12\sqrt{6})}{(-4c^2+4c\sqrt{6}-12)\sqrt{3}} \quad <&\quad  1 
\end{align}

The graph of function $g(c)=\frac{(24c-12\sqrt{6})}{(-4c^2+4c\sqrt{6}-12)\sqrt{3}} $ can be found in Figure~\ref{fig:matlab2}.


For  $c=\sqrt{6}-\sqrt{3}\approx 0.7174$ we obtain $g(c)=1$. Consequently, the inequality in  (\ref{eqn:th22inequal})
is not satisfied, and we have a contradiction to  (\ref{equ:E1assumption}).
\begin{figure}[ht]
\centering
\begin{minipage}{0.75\textwidth}\centering
\includegraphics[totalheight=6.0cm]{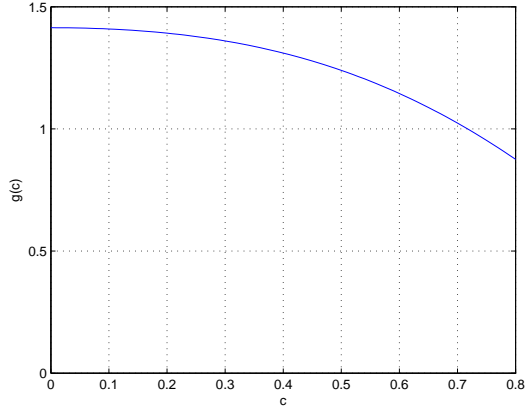}
\caption{The graph of function $g(c)=\frac{(24c-12\sqrt{6})}{(-4c^2+4c\sqrt{6}-12)\sqrt{3}} $}
\label{fig:matlab2}
\end{minipage}
\end{figure}
\end{proof}

\begin{proof}[Proof of Theorem~\ref{thm:finite}.\ref{thm:finite3} and~\ref{thm:finite}.\ref{thm:finite5}]
Analogous to the proof of the inequality in Theorem~\ref{thm:infinitegrid}.\ref{lab:thm:infinitegrid3},
we can deduce Theorem~\ref{thm:finite}.\ref{thm:finite3} and~\ref{thm:finite}.\ref{thm:finite5} 
from Theorem~\ref{thm:finite}.\ref{thm:finite1} and~\ref{thm:finite}.\ref{thm:finite4}, respectively.
\end{proof}

\section{Concluding Remarks}
It is quite reasonable to assume, that the constants $c$ in Theorem~\ref{thm:finite} are not tight. 
That is why, we pose the following conjecture: 
\begin{conjecture}
For any finite subset $W\subseteq V(G_r)$ with $|W|\leq 3r^2$ we have
	\begin{itemize}
		\item $|N(W)|,|\CE(W)|\geq c\cdot \sqrt{|W|}$ for $c=\frac2{\sqrt{3}}\approx 1.1547$.
	\end{itemize}
\end{conjecture}

%

\bibliography{grids}
\bibliographystyle{abbrv}

\end{document}